\documentclass[12pt]{article}
\usepackage[a4paper, margin=1.1in]{geometry}
\usepackage{amsmath,amsthm}
\usepackage{amsfonts,amssymb}
\usepackage{graphicx}
\usepackage{breqn}
\usepackage{amsmath}
\usepackage{float}
\usepackage{caption}
\usepackage{chngcntr}
\usepackage{fncylab}
\usepackage{verbatim}
\usepackage{cite}
\newtheorem{thm}{Theorem}[section]
\newtheorem{cor}[thm]{Corollary}
\newtheorem{lem}[thm]{Lemma}

\newtheorem{defn}{Definition}[section]
\theoremstyle{remark}

\newtheorem{exam}{Example}[section]

\begin{document}

\title
{\bf{Normalized Laplacian spectra of central vertex join and central edge join of graphs}}
\author {\small Jahfar T K \footnote{jahfartk@gmail.com} and Chithra A V \footnote{chithra@nitc.ac.in} \\ \small Department of Mathematics, National Institute of Technology, Calicut, Kerala, India-673601}
\date{ }
\maketitle
\begin{abstract}
	In this paper, we compute normalized Laplacian spectra of central graph of a  regular graph, central vertex join and central edge join of two regular graphs. Also, we determine their Kemeny's constant and degree Kirchhoff index. 
\end{abstract}

\hspace{-0.6cm}\textbf{AMS classification}: 05C50

\section{Introduction}
Let $G=(V(G),E(G))$ be  a simple undirected graph with vertex set $V(G)=\{v_{1}, v_{2},...,v_{n}\}$ and its edge set $E(G)=\{e_{1}, e_{2},...,e_{m}\}.$  The degree $d_i$ of the vertex $v_i\in V(G)$ is the number of vertices adjacent to $v_i$. The adjacency matrix $A(G)=[a_{ij}]$ of the graph $G$ is a square symmetric matrix of order $n$  whose $(i,j)^{th}$ entry is defined as,
$$a_{ij}=\begin{cases}
1,\text{if $v_i$ and $v_j$ are adjacent in $G$},\\
0,  \text{otherwise}.
\end{cases}.$$The Laplacian matrix of $G$, denoted by $L(G)$ is defined as $L(G)=D(G)-A(G)$, where $D(G)$ is the diagonal matrix with vertex degrees. The characteristic polynomial of the $n\times n$ matrix $M$ of $G$ is defined as $f(M,x)=|xI_n-M|$, where $I_n$ is the identity matrix of order $n$.  Eigenvalues of $A(G)$ are denoted by $\lambda_1(G)\geqslant\lambda_2(G)\geqslant...\geqslant\lambda_n(G).$ Eigenvalues of $A(G)$ together with their multiplicities are called the spectrum of $G$.
 The normalized Laplacian matrix $\mathcal{L}(G)=(\mathcal{L}_{ij})$ of $G$ is a square symmetric matrix of order $n$ whose $(i,j)^{th}$ entry is defined as,
\[\mathcal{L}_{ij}=\begin{cases}
1, \text{if $v_i=v_j$ and $d_i\ne 0$},\\
-\frac{1}{\sqrt{d_id_j}},\text{if $v_i$ and $v_j$ are adjacent in $G$},\\
0,  \text{otherwise}.
\end{cases}.\] The normalized Laplacian matrix $\mathcal{L}(G)$ of $G$ is defined as   $$ \mathcal{L}(G)=I_n- D^{-\frac{1}{2}}(G)A(G)D^{-\frac{1}{2}}(G).$$ 
The eigenvalues of the matrix $\mathcal{L}(G)$ are called the normalized Laplacian eigenvalues of $G$, and will be denoted by $0=\tilde{\mu}_1(G)<\tilde{\mu}_2(G)\leq ...\leq\tilde{\mu}_n(G).$ The eigenvalues of $\mathcal{L}(G)$ together with their multiplicities are called the normalized Laplacian spectrum of $G$. If $G$ is an $r$-regular graph, then the adjacency and the normalized Laplacian eigenvalues of $G$ are related by $\tilde{\mu}_i(G)=1-\frac{\lambda_i(G)}{r}$\cite{gutman2014randic,bozkurt2010randic}. Two graphs are said to be $\mathcal{L}$-cospectral, if they have the same $\mathcal{L}$-spectrum. Otherwise, they are non $\mathcal{L}$-cospectral graphs.
 In 1993, Klein and Randi{\'{c}} \cite{klein1993resistance} introduced the resistance distance between two vertices $v_i$ and $v_j$ in $G$, denoted by $r_{i,j}^*$, is defined as the effective resistance between them when unit resistors are distributed on every
edge of G. The degree  Kirchhoff index of a connected graph $G$ is defined by Chen et al. in \cite{chen2007resistance} as $$Kf^*(G)=\sum_{i<j}^{}d_ir_{i,j}^*d_j,$$ where $r_{i,j}^*$ denotes the resistance distance between vertices $v_i$ and $v_j$ in  $G$.  In \cite{chen2007resistance}, the authors proved that $$ Kf^*(G)=2m\sum_{i=2}^{n}\frac{1}{\tilde{\mu}_i(G)}$$
The Kemeny's constant $K(G)$ of a graph $G$ \cite{butler2016algebraic,hunter2014role} is defined as the expected number of steps from a randomly chosen starting point in the graph to a randomly chosen destination point in the graph. It is defined in terms of the  normalized Laplacian eigenvalues of $G$ as $$K(G)=\sum_{i=2}^{n}\frac{1}{\tilde{\mu}_i(G)}.$$
The various  application of Kemeny's constant
to perturbed Markov chains, random walks on directed graphs are studied in \cite{hunter2014role}.
We recall that for two matrices $A$ and $B$, of the same size, the Hadamard product  $A \circ B$ of $A$ and $B$ is a matrix of the same size with entries given by $(A \circ B)_{i j} =(A)_{i j}.(B)_{i j}$  (that is entry wise multiplication)\cite{horn1991topics}. The incidence matrix of $G$, $I(G)$ is the $n\times m$ matrix whose $(i,j)^{th}$ entry is 1 if $v_i$ is incident to $e_j$ and 0 otherwise.
  It is  known \textnormal{\cite{cvetkovic1980spectra}} that, $I(G)I(G)^T=A(G)+D(G)$. In particular, $G$ is an $r$-regular graph, then  $I(G)I(G)^T=A(G)+rI_n$. Also,  $I(G)I(G)^T=2rI_{n}-r\mathcal{L}(G)$\cite{cvetkovic2010introduction}. The adjacency matrix of the complement of a graph $G$ is $A(\bar{G}) =J_n-I_n-A(G)$. Throughout the paper for any positive integer m, $I_m$ denotes the $m\times m$  identity matrix, $J_m$ denotes the $m\times m$ matrix whose all entries equal to one and  $J_{m\times 1}$ denotes the column vector of size $m$ with all entries equal to one. For any positive integers $p$ and $q$, $O_{p\times q}$ denotes the zero matrix of size $p\times q$ and $K_{p,p}$ is the complete bipartite graph with $2p$ vertices.
 \begin{defn}\textnormal{\cite{vivin2008harmonious}}
	Let $G$ be a simple graph with $n$ vertices and $m$ edges. The central graph of $G$, denoted by $C(G)$  is obtained by subdividing each edge of $G$ exactly once and joining all the nonadjacent vertices in $G$.\\ The number of vertices and edges in $C(G)$ are $m+n$ and $m+\frac{n(n-1)}{2}$ respectively.
\end{defn}
\par The rest of the paper is organized as follows. In Section 2, we present necessary  definitions and known results, which will be used later. In Section 3, we calculate the normalized Laplacian spectrum of central graph of regular graph. In Section 4, normalized Laplacian spectra of central vertex join and central edge join of two regular graphs are obtained. In Section 5, some $\mathcal{L}$-cospectral non-regular graphs are constructed. Also, we investigate some graph invariants like the degree Kirchhoff index and the Kemeny's constant   of central vertex join and central edge join of two regular graphs.
 \section{Preliminaries}
In this section, some definitions and results are given  which will be useful to prove our main results. 
\begin{lem}\textnormal{\cite{cvetkovic1980spectra}}
	Let $U,V,W$ and $X$ be matrices with $U$  invertible. Let \[S=\begin{pmatrix}
	U&V\\
	W&X
	\end{pmatrix}. 
	\]  Then $det(S)=det(U)det(X-WU^{-1}V)$.\\ If $X$ is invertible,  then $det(S)=det(X)det(U-VX^{-1}W).$ If $U$ and $W$ are commutes, then $det(S)=det(UX-WV)$.
\end{lem}
\begin{lem}\textnormal{\cite{cvetkovic1980spectra}}
	Let $G$ be a connected $r$-regular graph on $n$ vertices with  adjacency matrix $A$ having $t$ distinct eigenvalues $r=\lambda_1,\lambda_2,...,\lambda_t$.  Then there exists a polynomial \[P(x)=n\dfrac{(x-\lambda_2)(x-\lambda_3)...(x-\lambda_t)}{(r-\lambda_2)(r-\lambda_3)...(r-\lambda_t)}.\] such that $P(A)=J_n, P(r)=n$ and $P(\lambda_i) =0$ for $\lambda_i\neq r.$
\end{lem}

\begin{defn}\textnormal{\cite{das2018spectra}}
	The $M$-coronal $\chi_M(x)$  of $n\times n$ matrix $M$ is defined as the sum of the entries of the matrix $(xI_n-M)^{-1}$ $($if exists$)$, that is, $$\chi_M(x)=J_{n\times 1}^T(xI_n-M)^{-1}J_{n\times 1}.$$
	Throughout this, we use the notation $\chi_M(x,\alpha)=J_{n\times 1}^T(xI_n-M\circ(\alpha J_n+(1-\alpha)I_n))^{-1}J_{n\times 1}$, as in \textnormal{\cite{tian2017normalized}}, if $\alpha = 1$, then $\chi_M(x, 1)$ is the usual
	M-coronal $\chi_M(x)$ . \\Let G be an $r$-regular graph of order n. Then
	$\chi_{\mathcal{L}(G)}(x,\alpha) = \frac{n}{x + \alpha-1}$ \textnormal{\cite{tian2017normalized}}.
\end{defn}
\begin{lem}\textnormal{\cite{mcleman2011spectra}}
	Let $G$ be an $r$-regular graph on $n$  vertices,  then $\chi_A(x)=\frac{n}{x-r}$ .
\end{lem}
For Laplacian matrix each row sum is zero, so $\chi_L(x)=\frac{n}{x}$.\\Note that $\chi_{kL}(x)= \chi_L(x)$ for any real number  k.
\begin{lem}\textnormal{\cite{liu2019spectra}}
	Let $A$  be an $n\times n$ real matrix. Then $det(A+\alpha J_{n})=det(A)+\alpha J_{n\times 1}^T adj(A)J_{n\times 1}$, where $\alpha $  is a real number and $adj(A)$  is the adjoint of $A$ . 
\end{lem}
\begin{cor}\textnormal{\cite{liu2019spectra}}
	Let $A$  be an $n\times n$  real matrix. Then \[det(xI_n-A-\alpha J_{n})=(1-\alpha \chi_A(x))det(xI_n-A).\] 
\end{cor}
\begin{lem}\textnormal{\cite{das2018spectra}}
	For any real numbers $c,d>0,(cI_n-dJ_{n})^{-1}= \frac{1}{c}I_n+\frac{d}{c(c-nd)}J_{n}.$
\end{lem} 
\begin{defn}\textnormal{\cite{jahfar2020central}}
	Let $G_1$ and $G_2$ be any two graphs on $n_1, n_2$ vertices and $m_1, m_2$ edges respectively. The central vertex join of $G_1$ and $G_2$ is the graph $G_1\dot{\vee} G_2$ is obtained from $C(G_1)$ and $G_2$  by joining each vertex of $G_1$  with every vertex of $G_2$.\\ The graph $G_1\dot{\vee} G_2$ has $m_1+n_1+n_2$ vertices and $m_1+m_2+n_1n_2+\frac{n_1(n_1-1)}{2}$ edges.
\end{defn}
\begin{defn}\textnormal{\cite{jahfar2020central}}
	Let $G_i$ be any two graphs with $n_i$ vertices and $m_i$ edges respectively for ${i=1,2}.$ Then the central edge join of two graphs $G_1$ and $G_2$ is the graph $G_1\veebar G_2 $ is obtained from $C(G_1)$ and $G_2$  by joining each vertex corresponding to edges of $G_1$  with every vertex of $G_2$.\\The graph $G_1\veebar G_2 $ has $m_1+n_1+n_2$ vertices and  $m_1+m_2+m_1n_2+\frac{n_1(n_1-1)}{2}$ edges.
\end{defn}
\section{Normalized Laplacian spectra of central graphs}
In this section, we determine the normalized Laplacian spectrum of central graph of regular graphs.
\begin{thm}
	Let $G$ be an r-regular graph with $n$ vertices. Then the normalized Laplacian characteristic polynomial  of central graph of $G$ is 
	\begin{align*}
	 f(\mathcal{L}(C(G)),x)= & (x-1)^{m-n} \prod_{i=1}^{n}\left((x-1)\Big((x-1)-\frac{P(\lambda_i)-1-\lambda_i}{n-1}\Big)-{\frac{(\lambda_i+r)}{2(n-1)}} \right ). 
	\end{align*}
\end{thm}
\begin{proof}
	Let $I(G)$ be the incidence matrix of $G$ and $A(\bar{G})$ be the adjacency matrix of complement of graph $G$. Then the adjacency matrix of $C(G)$ is 
	\[A(C(G))=\begin{pmatrix}
	A(\bar G)&I(G)\\
	I(G)^T&O_{m\times m}\\
	\end{pmatrix}. 
	\] The  degree matrix of $C(G)$ is
	\[D(C(G))=\begin{pmatrix}
	(n-1)I_n&O_{n\times m}\\
	O_{m\times n}&2I_m\\
	\end{pmatrix}. 
	\] Then the normalized Laplacian  matrix of $C(G)$ is \[\mathcal{L}(C(G))=I_{m+n}-D^{-\frac{1}{2}}(C(G))A(C(G))D^{-\frac{1}{2}}(C(G))=\begin{pmatrix}
	I_n-\frac{A(\bar G)}{n-1}&-\frac{I(G)}{\sqrt{2(n-1)}}\\
	-\frac{I(G)^T}{\sqrt{2(n-1)}}&I_{m}\\
	\end{pmatrix}. 
	\]Using Lemmas 2.1 and 2.2, we have
	 the normalized Laplacian characteristic polynomial of $C(G)$ is 
\footnotesize 
\begin{align*}
f(\mathcal{L}(C(G)),x)=& det\begin{pmatrix}
(x-1)I_n+\frac{J_n-I_n-A(G)}{n-1}&\frac{I(G)}{\sqrt{2(n-1)}}\\
\frac{I(G)^T}{\sqrt{2(n-1)}}&(x-1)I_m\\
\end{pmatrix} \\ 
= & (x-1)^m det\left((x-1)I_n+\frac{J_n-I_n-A(G)}{n-1}-\frac{I(G)I(G)^T}{2(x-1)(n-1)} \right )\\ 
= & (x-1)^{m-n} det\left((x-1)\Big((x-1)I_n+\frac{J_n-I_n-A(G)}{n-1}\Big)-{\frac{I(G)I(G)^T}{2(n-1)}} \right )\\
= & (x-1)^{m-n} det\left((x-1)\Big((x-1)I_n+\frac{J_n-I_n-A(G)}{n-1}\Big)-{\frac{(A(G)+rI_n)}{2(n-1)}} \right )\\
f(\mathcal{L}(C(G)),x)= & (x-1)^{m-n} \prod_{i=1}^{n}\left((x-1)\Big((x-1)+\frac{P(\lambda_i)-1-\lambda_i}{n-1}\Big)-{\frac{(\lambda_i+r)}{2(n-1)}} \right).
\end{align*}
\end{proof}
The following corollary describes the complete normalized Laplacian spectrum of $C(G)$  when $G$ is  regular graph.
\begin{cor}
Let $G$ be an r-regular graph on $n$ vertices.  Then the normalized Laplacian spectrum of $C(G)$  consists of 
	\begin{enumerate}
		\item $1$ with multiplicity $m-n$ .
		\item Two roots of the equation $(x-1)\Big((x-1)+\frac{n-1-r}{n-1}\Big)-{\frac{r}{(n-1)}}=0.$
    	\item Two roots of the equation $(x-1)\Big((x-1)+\frac{-1-\lambda_i}{n-1}\Big)-{\frac{(\lambda_i+r)}{2(n-1)}}=0$  \textnormal{for} $i=2,...,n.$
	\end{enumerate}
\end{cor}
\begin{exam}
	Let $G=K_{p,p}$. Then the adjacency  eigenvalues of $G$ are 0 (multiplicity $2p-2$) and $\pm p$. Therefore, the normalized Laplacian eigenvalues of $C(G)$ are 0, $\frac{3p-1}{2p-1}$,1(multiplicity $p^2-2p$),  roots of the equation $(2p-1)x^2+(1-3p)x+p=0$ and roots of the equation $(4p-2)x^2-(8p-2)x-3p=0$ (each root with multiplicity $2p-2$).\\
	In particular, $G=K_{3,3}$ . The normalized Laplacian eigenvalues of $C(G)$ are 0 ,  $\frac{8}{5}$,1(multiplicity 3), roots of the equation $5x^2-8x+3=0$ and roots of the equation $10x^2-22x+9=0$ (each root with multiplicity 4). 
\end{exam}	 
\section{Normalized Laplacian spectra of central vertex join and central edge join of Graphs}
In this section, we determine the normalized Laplacian spectrum of central vertex join and central edge join of two regular graphs.
\begin{thm}
	Let $G_i$ be an $r_i$-regular graph with $n_i$ vertices and $m_i$ edges for i=1,2. Then the normalized Laplacian characteristic polynomial of $G_1\dot{\vee} G_2$ is given by
	\begin{align*}
	 &f(\mathcal{L}(G_1\dot{\vee} G_2),x)\\ =&(x-1)^{m_1-n_1}\prod_{i=2}^{n_2}\left((x-\frac{n_1}{r_2+n_1})I_{n_2} -\frac{r_2}{r_2+n_1}\tilde{\mu}_i(G_2)\right)\\
	 &\times\bigg[\left(x-\frac{n_1}{r_2+n_1}\right)\left((x-1)^2-\frac{r_1}{(n_1+n_2-1)}-\frac{(1+r_1)(x-1)}{n_1+n_2-1}\right)\\&\,\,\,\,\,\,\,\,\,\,\,\,\,\,\,\,\,\,\,\,\;\;\;\;\;\;\;\;\;\;\;\;\;\;\;\;\;\;\;\;\;\;\;\;\;\;\;\;\;\;\;\;\;\;\;\;\;\;\;\;\;\;\;\;\;\;\:-\frac{n_1n_2(x-1)-n_1(x-1)(xn_1+xr_2-n_1)}{(n_1+n_2-1)(n_1+r_2)}\bigg]\\&\times\prod_{i=2}^{n_1}\left((x-1)^2-\frac{r_1}{(n_1+n_2-1)}-\frac{(1+r_1)(x-1)}{n_1+n_2-1}+\frac{r_1\tilde{\mu}_i(G_1)}{2(n_1+n_2-1)}+\frac{r_1(x-1)\tilde{\mu}_i(G_1)}{n_1+n_2-1}\right).\\
	\end{align*}
\end{thm}
\begin{proof}
	Let $I(G_1)$ be the incidence matrix of $G_1.$ Then by a proper labeling of vertices, the adjacency matrix of $G_1\dot{\vee} G_2$ can be written as 
	\begin{align*}
    A(G_1\dot{\vee} G_2)=\begin{pmatrix}
     A(\bar {G_1})&I(G_1)&J_{n_1\times n_2}\\
     I(G_1)^T&O_{m_1 \times m_1}& O_{{m_1}\times {n_2}}\\
     J_{n_2\times n_1}&O_{n_2\times m_1}&A(G_2) 
     \end{pmatrix}. \\
     \end{align*}
     The degree matrix of $G_1\dot{\vee} G_2$ is
    \begin{align*}
     D(G_1\dot{\vee} G_2)=\begin{pmatrix}
     (n_1+n_2-1)I_{n_1}&O_{n_1\times m_1}&O_{n_1\times n_2}\\
     O_{m_1\times n_1}&2I_{m_1}& O_{{m_1}\times {n_2}}\\
     O_{n_2\times n_1}&O_{n_2\times m_1}&(r_2+n_1)I_{n_2} 
     \end{pmatrix}. \\
     \end{align*}
      \begin{align*}
     D^{-\frac{1}{2}}(G_1\dot{\vee} G_2)=\begin{pmatrix}
     \frac{I_{n_1}}{\sqrt{(n_1+n_2-1)}}&O_{n_1\times m_1}&O_{n_1\times n_2}\\
     O_{m_1\times n_1}&\frac{I_{m_1}}{\sqrt{2}}& O_{{m_1}\times {n_2}}\\
     O_{n_2\times n_1}&O_{n_2\times m_1}&\frac{I_{n_2}}{\sqrt{(r_2+n_1)}} 
     \end{pmatrix}. \\
     \end{align*}
     Then the normalized Laplacian matrix of $G_1\dot{\vee} G_2$ is 
     \begin{align*}
    \mathcal{L}(G_1\dot{\vee} G_2)=\begin{pmatrix}
    I_{n_1}-\frac{ A(\bar {G_1})}{n_1+n_2-1}&-\frac{I(G_1)}{\sqrt{2(n_1+n_2-1)}}&-\frac{J_{n_1\times n_2}}{\sqrt{(n_1+n_2-1)(n_1+r_2)}}\\
     -\frac{I(G_1)^T}{\sqrt{2(n_1+n_2-1)}}&I_{m_1 }& O_{{m_1}\times {n_2}}\\
     -\frac{J_{n_2\times n_1}}{\sqrt{(n_1+n_2-1)(n_1+r_2)}}&O_{{n_2}\times {m_1}}&I_{n_2}-\frac{A(G_2) }{(n_1+r_2)}
     \end{pmatrix} \\
     \end{align*}
     \begin{align*}
     \mathcal{L}(G_1\dot{\vee} G_2)=\begin{pmatrix}
     I_{n_1}-\frac{ A(\bar {G_1})}{n_1+n_2-1}&-\frac{I(G_1)}{\sqrt{2(n_1+n_2-1)}}&-\frac{J_{n_1\times n_2}}{\sqrt{(n_1+n_2-1)(n_1+r_2)}}\\
     -\frac{I(G_1)^T}{\sqrt{2(n_1+n_2-1)}}&I_{m_1 }& O_{{m_1}\times {n_2}}\\
     -\frac{J_{n_2\times n_1}}{\sqrt{(n_1+n_2-1)(n_1+r_2)}}&O_{{n_2}\times {m_1}}&\mathcal{L}(G_2)\circ B
     \end{pmatrix}, \\
     \end{align*}
     where $B=\frac{r_2}{r_2+n_1}J_{n_2 }+\frac{n_1}{r_2+n_1}I_{n_2}.$\\Using Definition 2.1, Lemmas 2.1, 2.2, 2.3 and Corollary 2.5, we have the normalized Laplacian characteristic polynomial of $G_1\dot{\vee} G_2$ is \\
     \begin{align*}
     f(\mathcal{L}(G_1\dot{\vee} G_2),x) =&  det\begin{pmatrix}
     (x-1)I_{n_1}+\frac{ A(\bar {G_1})}{n_1+n_2-1}&\frac{I(G_1)}{\sqrt{2(n_1+n_2-1)}}&\frac{J_{n_1\times n_2}}{\sqrt{(n_1+n_2-1)(n_1+r_2)}}\\
     \frac{I(G_1)^T}{\sqrt{2(n_1+n_2-1)}}&(x-1)I_{m_1}& O_{{m_1}\times {n_2}}\\
     \frac{J_{n_2\times n_1}}{\sqrt{(n_1+n_2-1)(n_1+r_2)}}&O_{{n_2}\times {m_1}}&xI_{n_2}-(\mathcal{L}(G_2)\circ B)
     \end{pmatrix} \\ 
     =&\det\bigg(xI_{n_2}-(\mathcal{L}(G_2)\circ B)\bigg)\det S,
     \end{align*}
     \begin{align*}\textnormal{where~S} =&\begin{pmatrix}
     (x-1)I_{n_1}+\frac{ A(\bar {G_1})}{n_1+n_2-1}&\frac{I(G_1)}{\sqrt{2(n_1+n_2-1)}}\\
      \frac{I(G_1)^T}{\sqrt{2(n_1+n_2-1)}}&{(x-1)I_m}_1\\
     \end{pmatrix}-\bigg[\begin{pmatrix}
     \frac{J_{n_1\times n_2}}{\sqrt{(n_1+n_2-1)(n_1+r_2)}}\\
     O_{{m_1}\times {n_2}}
     \end{pmatrix}\left(xI_{n_2}-(\mathcal{L}(G_2)\circ B)\right)^{-1}\\&\,\,\,\,\,\,\,\,\,\,\,\,\,\,\,\,\,\,\,\,\;\;\;\;\;\;\;\;\;\;\;\;\;\;\;\;\;\;\;\;\;\;\;\;\;\;\;\;\;\;\;\;\;\;\;\;\;\;\;\;\;\;\;\;\;\;\:\times\begin{pmatrix}
     \frac{J_{n_2\times n_1}}{\sqrt{(n_1+n_2-1)(n_1+r_2)}}&O_{{n_2}\times {m_1}}\end {pmatrix}\bigg]
      \end{align*}
      \begin{align*}
     =&\begin{pmatrix}
     (x-1)I_{n_1}+\frac{ A(\bar {G_1})}{n_1+n_2-1}-\frac{J_{n_1\times n_2}}{\sqrt{(n_1+n_2-1)(n_1+r_2)}}\left(xI_{n_2}-(\mathcal{L}(G_2)\circ B)\right)^{-1}\frac{J_{n_2\times n_1}}{\sqrt{(n_1+n_2-1)(n_1+r_2)}}&\frac{I(G_1)}{\sqrt{2(n_1+n_2-1)}}\\
     \frac{I(G_1)^T}{\sqrt{2(n_1+n_2-1)}}&{(x-1)I_m}_1\\
     \end{pmatrix}\\
     =&\begin{pmatrix}
     (x-1)I_{n_1}+\frac{ A(\bar {G_1})}{n_1+n_2-1}-\frac{\chi_{\mathcal{L}(G_2)}(x,\frac{r_2}{r_2+n_1})}{(n_1+n_2-1)(n_1+r_2)}J_{n_1}&\frac{I(G_1)}{\sqrt{2(n_1+n_2-1)}}\\
     \frac{I(G_1)^T}{\sqrt{2(n_1+n_2-1)}}&{(x-1)I_m}_1
     \end{pmatrix}.
      \end{align*}Therefore,
      \begin{align*}
     det&(S)=det\begin{pmatrix}
     (x-1)I_{n_1}+\frac{ A(\bar {G_1})}{n_1+n_2-1}-\frac{\chi_{\mathcal{L}(G_2)}(x,\frac{r_2}{r_2+n_1})}{(n_1+n_2-1)(n_1+r_2)}J_{n_1}&\frac{I(G_1)}{\sqrt{2(n_1+n_2-1)}}\\
     \frac{I(G_1)^T}{\sqrt{2(n_1+n_2-1)}}&{(x-1)I_m}_1\\
     \end{pmatrix}\\
     =&(x-1)^{m_1}det\left(
     (x-1)I_{n_1}+\frac{ A(\bar {G_1})}{n_1+n_2-1}-\frac{\chi_{\mathcal{L}(G_2)}(x,\frac{r_2}{r_2+n_1})}{(n_1+n_2-1)(n_1+r_2)}J_{n_1}-\frac{I(G_1)I(G_1)^T}{2(n_1+n_2-1)(x-1)}\right)\\
     =&(x-1)^{m_1}det\bigg(
     (x-1)I_{n_1}+\frac{ A(\bar {G_1})}{n_1+n_2-1}-\frac{\chi_{\mathcal{L}(G_2)}(x,\frac{r_2}{r_2+n_1})}{(n_1+n_2-1)(n_1+r_2)}J_{n_1}\\&    -\frac{r_1}{(n_1+n_2-1)(x-1)}I_{n_1}+\frac{r_1\mathcal{L}(G_1)}{2(n_1+n_2-1)(x-1)}\bigg)\\
      \end{align*}
     \begin{align*}
      =&(x-1)^{m_1}det\bigg(
     (x-1)I_{n_1}+\frac{ J_{n_1}-I_{n_1}-A(G_1) }{n_1+n_2-1}-\frac{\chi_{\mathcal{L}(G_2)}(x,\frac{r_2}{r_2+n_1})}{(n_1+n_2-1)(n_1+r_2)}J_{n_1}\\&    -\frac{r_1}{(n_1+n_2-1)(x-1)}I_{n_1}+\frac{r_1\mathcal{L}(G_1)}{2(n_1+n_2-1)(x-1)}\bigg)\\
      =&(x-1)^{m_1}det\bigg(
     (x-1)I_{n_1}+\frac{ J_{n_1}-I_{n_1}-r_1I_{n_1}+r_1\mathcal{L}(G_1) }{n_1+n_2-1}-\frac{\chi_{\mathcal{L}(G_2)}(x,\frac{r_2}{r_2+n_1})}{(n_1+n_2-1)(n_1+r_2)}J_{n_1}\\&    -\frac{r_1}{(n_1+n_2-1)(x-1)}I_{n_1}+\frac{r_1\mathcal{L}(G_1)}{2(n_1+n_2-1)(x-1)}\bigg)\\
       =&(x-1)^{m_1}\Bigg(1-\left(\frac{\chi_{\mathcal{L}(G_2)}(x,\frac{r_2}{r_2+n_1})}{(n_1+n_2-1)(n_1+r_2)}-\frac{1}{n_1+n_2-1}\right)\\&\chi_{k\mathcal{L}(G_1)}\left(x-1-\frac{r_1}{(n_1+n_2-1)(x-1)}-\frac{1+r_1}{n_1+n_2-1}\right)\Bigg)\\&\times det\left(\left(x-1-\frac{r_1}{(n_1+n_2-1)(x-1)}-\frac{1+r_1}{n_1+n_2-1}\right)I_{n_1}+\frac{r_1\mathcal{L}(G_1)}{2(n_1+n_2-1)(x-1)}+\frac{r_1\mathcal{L}(G_1)}{n_1+n_2-1}\right)
      \bigg),\end{align*} where $k=-\frac{r_1}{2(n_1+n_2-1)(x-1)}-\frac{r_1}{n_1+n_2-1}.$\begin{align*}
    \end{align*}
     It is easy to see that \begin{align*}
     &det\bigg(xI_{n_2}-(\mathcal{L}(G_2)\circ B)\bigg)=det\left((x-\frac{n_1}{r_2+n_1})I_{n_2} -\frac{r_2}{r_2+n_1}\mathcal{L}(G_2)\right).\\
     &\chi_{\mathcal{L}(G_2)}\left(x,\frac{r_2}{r_2+n_1}\right)=\frac{n_2}{x-1+\frac{r_2}{r_2+n_1}}.\\&\chi_{k\mathcal{L}(G_1)}\left(x-1-\frac{r_1}{(n_1+n_2-1)(x-1)}-\frac{1+r_1}{n_1+n_2-1}\right)=\frac{n_1}{x-1-\frac{r_1}{(n_1+n_2-1)(x-1)}-\frac{1+r_1}{n_1+n_2-1}}. \end{align*}Therefore,
     \begin{align*}
      &det(S) =(x-1)^{m_1-n_1}\Bigg(1-\left(\frac{\chi_{\mathcal{L}(G_2)}(x,\frac{r_2}{r_2+n_1})}{(n_1+n_2-1)(n_1+r_2)}-\frac{1}{n_1+n_2-1}\right)\\&\times\left(\frac{n_1}{x-1-\frac{r_1}{(n_1+n_2-1)(x-1)}-\frac{1+r_1}{n_1+n_2-1}}\right)\Bigg)\\&\times det\left(\left((x-1)^2-\frac{r_1}{(n_1+n_2-1)}-\frac{(1+r_1)(x-1)}{n_1+n_2-1}\right)I_{n_1}+\frac{r_1\mathcal{L}(G_1)}{2(n_1+n_2-1)}+\frac{r_1(x-1)\mathcal{L}(G_1)}{n_1+n_2-1}\right)
     \bigg).\\
     &f(\mathcal{L}(G_1\dot{\vee} G_2),x) =\det\bigg(xI_{n_2}-(\mathcal{L}(G_2)\circ B)\bigg)\det S\\
     &\;\;\;\;\;\;\;\;\;\;\;\;\;\;\;\;\;\;\;\;\;\;\;\;\;=det\left((x-\frac{n_1}{r_2+n_1})I_{n_2} -\frac{r_2}{r_2+n_1}\mathcal{L}(G_2)\right)\\
     &\times(x-1)^{m_1-n_1}\bigg[1-\left(\frac{\frac{n_2}{x-1+\frac{r_2}{r_2+n_1}}}{(n_1+n_2-1)(n_1+r_2)}-\frac{1}{n_1+n_2-1}\right)\\&\,\,\,\,\,\,\,\,\,\,\,\,\,\,\,\,\,\,\,\,\;\;\;\;\;\;\;\;\;\;\;\;\;\;\;\;\;\;\;\;\;\;\;\;\;\;\;\;\;\;\;\;\;\;\;\;\;\;\;\;\;\;\;\;\;\;\:\times\left(\frac{n_1}{x-1-\frac{r_1}{(n_1+n_2-1)(x-1)}-\frac{1+r_1}{n_1+n_2-1}}\right)\\&\times det\left((x-1)^2-\frac{r_1}{(n_1+n_2-1)}-\frac{(1+r_1)(x-1)}{n_1+n_2-1}+\frac{r_1\mathcal{L}(G_1)}{2(n_1+n_2-1)}+\frac{r_1(x-1)\mathcal{L}(G_1)}{n_1+n_2-1}\right)
     \bigg]\\
     =&\prod_{i=1}^{n_2}\left((x-\frac{n_1}{r_2+n_1})I_{n_2} -\frac{r_2}{r_2+n_1}\tilde{\mu}_i(G_1)\right)\\
     &\times(x-1)^{m_1-n_1}\bigg[1-\left(\frac{n_1n_2(x-1)-n_1(x-1)(xn_1+xr_2-n_1)}{(x-1+\frac{r_2}{r_2+n_1})(n_1+n_2-1)(n_1+r_2)\left((x-1)^2-\frac{r_1}{(n_1+n_2-1)}-\frac{(1+r_1)(x-1)}{n_1+n_2-1}\right)}\right)\bigg]\\&\prod_{i=1}^{n_1}\left((x-1)^2-\frac{r_1}{(n_1+n_2-1)}-\frac{(1+r_1)(x-1)}{n_1+n_2-1}+\frac{r_1\tilde{\mu}_i(G_1)}{2(n_1+n_2-1)}+\frac{r_1(x-1)\tilde{\mu}_i(G_1)}{n_1+n_2-1}\right)
     \\
     =&(x-1)^{m_1-n_1}\prod_{i=2}^{n_2}\left((x-\frac{n_1}{r_2+n_1})I_{n_2} -\frac{r_2}{r_2+n_1}\tilde{\mu}_i(G_2)\right)\\
     &\times\bigg[\left(x-1+\frac{r_2}{r_2+n_1}\right)\left((x-1)^2-\frac{r_1}{(n_1+n_2-1)}-\frac{(1+r_1)(x-1)}{n_1+n_2-1}\right)\\&\,\,\,\,\,\,\,\,\,\,\,\,\,\,\,\,\,\,\,\,\;\;\;\;\;\;\;\;\;\;\;\;\;\;\;\;\;\;\;\;\;\;\;\;\;\;\;\;\;\;\;\;\;\;\;\;\;\;\;\;\;\;\;\;\;\;\:-\frac{n_1n_2(x-1)-n_1(x-1)(xn_1+xr_2-n_1)}{(n_1+n_2-1)(n_1+r_2)}\bigg]\\&\prod_{i=2}^{n_1}\left((x-1)^2-\frac{r_1}{(n_1+n_2-1)}-\frac{(1+r_1)(x-1)}{n_1+n_2-1}+\frac{r_1\tilde{\mu}_i(G_1)}{2(n_1+n_2-1)}+\frac{r_1(x-1)\tilde{\mu}_i(G_1)}{n_1+n_2-1}\right)
     \\\end{align*}
     \begin{align*}
      =&(x-1)^{m_1-n_1}\prod_{i=2}^{n_2}\left((x-\frac{n_1}{r_2+n_1})I_{n_2} -\frac{r_2}{r_2+n_1}\tilde{\mu}_i(G_2)\right)\\
     &\times\bigg[\left(x-\frac{n_1}{r_2+n_1}\right)\left((x-1)^2-\frac{r_1}{(n_1+n_2-1)}-\frac{(1+r_1)(x-1)}{n_1+n_2-1}\right)\\&\,\,\,\,\,\,\,\,\,\,\,\,\,\,\,\,\,\,\,\,\;\;\;\;\;\;\;\;\;\;\;\;\;\;\;\;\;\;\;\;\;\;\;\;\;\;\;\;\;\;\;\;\;\;\;\;\;\;\;\;\;\;\;\;\;\;\:-\frac{n_1n_2(x-1)-n_1(x-1)(xn_1+xr_2-n_1)}{(n_1+n_2-1)(n_1+r_2)}\bigg]\\&\prod_{i=2}^{n_1}\left((x-1)^2-\frac{r_1}{(n_1+n_2-1)}-\frac{(1+r_1)(x-1)}{n_1+n_2-1}+\frac{r_1\tilde{\mu}_i(G_1)}{2(n_1+n_2-1)}+\frac{r_1(x-1)\tilde{\mu}_i(G_1)}{n_1+n_2-1}\right).
     \\
 \end{align*}
   \end{proof}
The following corollary describes the complete normalized Laplacian spectrum of $G_1\dot{\vee} G_2$ when $G_1$ and $G_2$ are regular graphs.
	\begin{cor}
	Let $G_i$ be an $r_i$-regular graph with $n_i$ vertices and $m_i$ edges for i=1,2. Then the normalized Laplacian spectrum of $G_1\dot{\vee} G_2$ consists of 
	\begin{enumerate}
		
		\item $1$ repeated $m_1-n_1$ times.
		\item Roots of the equation,  $$(x-\frac{n_1}{r_2+n_1}) -\frac{r_2}{r_2+n_1}\tilde{\mu}_i(G_2)=0 ~\textnormal{for}~ i=2,...,n_2.$$
		\item Three roots of the equation,  \begin{align*}&\left(x-\frac{n_1}{r_2+n_1}\right)\left((x-1)^2-\frac{r_1}{n_1+n_2-1}-\frac{(1+r_1)(x-1)}{n_1+n_2-1}\right)\\&-\left(\frac{n_1n_2(x-1)-n_1(x-1)(xn_1+xr_2-n_1)}{(n_1+n_2-1)(n_1+r_2)}\right)=0.\end{align*}
		\item Two roots of the equation, $$(x-1)^2-\frac{r_1}{(n_1+n_2-1)}-\frac{(1+r_1)(x-1)}{n_1+n_2-1}+\frac{r_1\tilde{\mu}_i(G_1)}{2(n_1+n_2-1)}+\frac{r_1(x-1)\tilde{\mu}_i(G_1)}{n_1+n_2-1}=0~\textnormal{for}~ i=2,..,n_1.$$
	\end{enumerate}
\end{cor}

\begin{exam}
	Let $G_1=K_{p,p}$ and  $G_2=K_2$ . Then the normalized Laplacian eigenvalues of $G_1$ are 0, 1(multiplicity $2p-2$) and 2. The normalized Laplacian eigenvalues of $G_2$ are 0 and 2. Therefore, the normalized Laplacian eigenvalues of $G_1\dot{\vee} G_2$ are 0, $\frac{2p+2}{2p+1}$,1(multiplicity $p^2-2p$), roots of the equation $(4p^2+4p+1)x^2-(10p^2+11p+3)x+(6p^2+6p+2)=0$, roots of the equation  $(4p+2)x^2-(6p+6)x+2p+4=0$ and roots of the equation $(4p+2)x^2-(8p+6)x+3p+4=0$ (each root with multiplicity $2p-2$).\\
In particular, $G_1=K_{3,3}$ and  $G_2=K_2$ . Then the normalized Laplacian eigenvalues of $G_1$ are 0, 1(multiplicity 4) and 2. The normalized Laplacian eigenvalues of $G_2$ are 0 and 2. Therefore, the normalized Laplacian eigenvalues of $G_1\dot{\vee} G_2$ are 0, $\frac{8}{7}$,1(multiplicity 3), roots of the equation $49x^2-126x+74=0$, roots of the equation $14x^2-24x+10=0$ and roots of the equation $14x^2-30x+13=0$ (each root with multiplicity 4).
\end{exam}

 Next we consider the normalized Laplacian spectrum of $G_1\veebar G_2$.
\begin{thm}
	Let $G_i$ be an $r_i$-regular graph with $n_i$ vertices and $m_i$ edges for i=1,2. Then the normalized Laplacian characteristic polynomial of $G_1\veebar G_2$ is given by
	\begin{align*}
&f(\mathcal{L}(G_1\veebar G_2),x)
=(x-1)^{m_1-n_1-1}\prod_{i=2}^{n_2}\left((x-\frac{m_1}{r_2+m_1})I_{n_2} -\frac{r_2}{r_2+m_1}\tilde{\mu}_i(G_2)\right)\\
&\prod_{i=2}^{n_1}\bigg[(x-1)^2+\frac{ \left(-1-r_1+r_1\tilde{\mu}_i(G_1)\right)(x-1)}{n_1-1}-\frac{(2r_1-r_1\tilde{\mu}_i(G_1))}{(n_1-1)(n_2+2)}\bigg]
\\&\times\bigg[\bigg((x-1)^2+\frac{(n_1-1-r_1)(x-1)}{n_1-1}-\frac{2r_1}{(n_1-1)(n_2+2)}\bigg)\\&\bigg((x-\frac{m_1}{r_2+m_1})(x-1)-\frac{n_2m_1}{(n_2+2)(r_2+m_1)}\bigg)-\frac{n_2r_1^2n_1}{(n_1-1)(n_2+2)^2(r_2+m_1)}\bigg].\\
\end{align*}
	
\end{thm}
\begin{proof}
	Let $I(G_1)$ be the incidence matrix of $G_1.$ Then by a proper labeling of vertices, adjacency matrix of $G_1\veebar G_2$ can be written as 
	\begin{align*}
	A(G_1\veebar G_2)=\begin{pmatrix}
	A(\bar {G_1})&I(G_1)&O_{{n_1}\times {n_2}}\\
	I(G_1)^T&O_{m_1 \times m_1}& J_{m_1\times n_2}\\
	O_{n_2\times n_1}&J_{n_2\times m_1}&A(G_2) 
	\end{pmatrix}. \\
	\end{align*}
	The degree matrix of $G_1\veebar G_2$ is
	\begin{align*}
	D(G_1\veebar G_2)=\begin{pmatrix}
	(n_1-1)I_{n_1}&O_{n_1\times m_1}&O_{n_1\times n_2}\\
	O_{m_1\times n_1}&(n_2+2)I_{m_1}& O_{{m_1}\times {n_2}}\\
	O_{n_2\times n_1}&O_{n_2\times m_1}&(r_2+m_1)I_{n_2} 
	\end{pmatrix}. \\
	\end{align*}
	\begin{align*}
	D^{-\frac{1}{2}}(G_1\veebar G_2)=\begin{pmatrix}
	\frac{I_{n_1}}{\sqrt{(n_1-1)}}&O_{n_1\times m_1}&O_{n_1\times n_2}\\
	O_{m_1\times n_1}&\frac{I_{m_1}}{\sqrt{n_2+2}}& O_{{m_1}\times {n_2}}\\
	O_{n_2\times n_1}&O_{n_2\times m_1}&\frac{I_{n_2}}{\sqrt{(r_2+m_1)}} 
	\end{pmatrix}. \\
	\end{align*}
	 Then the normalized Laplacian matrix of $G_1\veebar G_2$ is 
	\begin{align*}
	\mathcal{L}(G_1\veebar G_2)=\begin{pmatrix}
	I_{n_1}-\frac{ A(\bar {G_1})}{n_1-1}&-\frac{I(G_1)}{\sqrt{(n_1-1)(n_2+2)}}&O_{{n_1}\times {n_2}}\\
	-\frac{I(G_1)^T}{\sqrt{(n_1-1)(n_2+2)}}&I_{m_1 }& -\frac{J_{m_1\times n_2}}{\sqrt{(n_2+2)(r_2+m_1)}}\\
    O_{{n_2}\times {m_1}}&-\frac{J_{n_2\times m_1}}{\sqrt{(n_2+2)(r_2+m_1)}}&I_{n_2}-\frac{A(G_2) }{(m_1+r_2)}
	\end{pmatrix} \\
	\end{align*}
	\begin{align*}
	\mathcal{L}(G_1\veebar G_2)=\begin{pmatrix}
	I_{n_1}-\frac{ A(\bar {G_1})}{n_1-1}&-\frac{I(G_1)}{\sqrt{(n_1-1)(n_2+2)}}&O_{{n_1}\times {n_2}}\\
	-\frac{I(G_1)^T}{\sqrt{(n_1-1)(n_2+2)}}&I_{m_1 }& -\frac{J_{m_1\times n_2}}{\sqrt{(n_2+2)(r_2+m_1)}}\\
	O_{{n_2}\times {m_1}}&-\frac{J_{n_2\times m_1}}{\sqrt{(n_2+2)(r_2+m_1)}}&\mathcal{L}(G_2)\circ C
	\end{pmatrix} \\
	\end{align*}
	where $C=\frac{r_2}{r_2+m_1}J_{n_2 }+\frac{m_1}{r_2+m_1}I_{n_2}$.\\ Using Definition 2.1, Lemmas 2.1, 2.2, 2.3, 2.6 and Corollary 2.5, we have the normalized Laplacian characteristic polynomial of $G_1\veebar G_2$ is\\
	\begin{align*}
	f(\mathcal{L}(G_1\veebar G_2),x) =&  det\begin{pmatrix}
	(x-1)I_{n_1}+\frac{ A(\bar {G_1})}{n_1-1}&\frac{I(G_1)}{\sqrt{(n_1-1)(n_2+2)}}&O_{{n_1}\times {n_2}}\\
	\frac{I(G_1)^T}{\sqrt{(n_1-1)(n_2+2)}}&(x-1)I_{m_1 }& \frac{J_{m_1\times n_2}}{\sqrt{(n_2+2)(r_2+m_1)}}\\
	O_{{n_2}\times {m_1}}&\frac{J_{n_2\times m_1}}{\sqrt{(n_2+2)(r_2+m_1)}}&xI_{n_2}-(\mathcal{L}(G_2)\circ C)
	\end{pmatrix} \\ 
	=&\det\bigg(xI_{n_2}-(\mathcal{L}(G_2)\circ C)\bigg)\det S,\\
	\textnormal{where~S} =&\begin{pmatrix}
	(x-1)I_{n_1}+\frac{ A(\bar {G_1})}{n_1-1}&\frac{I(G_1)}{\sqrt{(n_1-1)(n_2+2)}}\\
	\frac{I(G_1)^T}{\sqrt{(n_1-1)(n_2+2)}}&{(x-1)I_m}_1\\
	\end{pmatrix}-\bigg[\begin{pmatrix}
	O_{{n_1}\times {n_2}}\\
	\frac{J_{m_1\times n_2}}{\sqrt{(n_2+2)(r_2+m_1)}}
	\end{pmatrix}\\&\,\,\,\,\,\,\,\,\,\,\,\,\,\,\,\,\,\,\,\,\;\;\;\;\;\;\;\;\;\;\;\;\;\;\;\;\;\;\;\;\;\;\;\;\;\;\:\times\left(xI_{n_2}-(\mathcal{L}(G_2)\circ C)\right)^{-1}\begin{pmatrix}
	O_{{n_2}\times {m_1}}&\frac{J_{n_2\times m_1}}{\sqrt{(n_2+2)(r_2+m_1)}}\end {pmatrix}\bigg]
	\end{align*}
	\begin{align*}
	S=&\begin{pmatrix}
	(x-1)I_{n_1}+\frac{ A(\bar {G_1})}{n_1-1}&\frac{I(G_1)}{\sqrt{(n_1-1)(n_2+2)}}\\
	\frac{I(G_1)^T}{\sqrt{(n_1-1)(n_2+2)}}&{(x-1)I_m}_1-\frac{J_{m_1\times n_2}}{\sqrt{(n_2+2)(r_2+m_1)}}\left(xI_{n_2}-(\mathcal{L}(G_2)\circ C)\right)^{-1}\frac{J_{n_2\times m_1}}{\sqrt{(n_2+2)(r_2+m_1)}}\\
	\end{pmatrix}\\
	=&\begin{pmatrix}
	(x-1)I_{n_1}+\frac{ A(\bar {G_1})}{n_1-1}&\frac{I(G_1)}{\sqrt{(n_1-1)(n_2+2)}}\\
	\frac{I(G_1)^T}{\sqrt{(n_1-1)(n_2+2)}}&{(x-1)I_m}_1-\frac{\chi_{\mathcal{L}(G_2)}(x,\frac{r_2}{r_2+m_1})}{(n_2+2)(r_2+m_1)}J_{m_1}\\
	\end{pmatrix}.\\\textnormal{Therefore,}\\
	det(S)=&det\begin{pmatrix}
	(x-1)I_{n_1}+\frac{ A(\bar {G_1})}{n_1-1}&\frac{I(G_1)}{\sqrt{(n_1-1)(n_2+2)}}\\
	\frac{I(G_1)^T}{\sqrt{(n_1-1)(n_2+2)}}&{(x-1)I_m}_1-\frac{\chi_{\mathcal{L}(G_2)}(x,\frac{r_2}{r_2+m_1})}{(n_2+2)(r_2+m_1)}J_{m_1}\\
	\end{pmatrix}
	\end{align*}
\begin{align*}
=&det\left((x-1){I_m}_1-\frac{\chi_{\mathcal{L}(G_2)}(x,\frac{r_2}{r_2+m_1})}{(n_2+2)(r_2+m_1)}J_{m_1}\right)det\bigg[
(x-1)I_{n_1}+\frac{ A(\bar {G_1})}{n_1-1}\\&-\frac{I(G_1)}{\sqrt{(n_1-1)(n_2+2)}}\left((x-1){I_m}_1-\frac{\chi_{\mathcal{L}(G_2)}(x,\frac{r_2}{r_2+m_1})}{(n_2+2)(r_2+m_1)}J_{m_1}\right)^{-1}\frac{I(G_1)^T}{\sqrt{(n_1-1)(n_2+2)}}\bigg]\\
=&(x-1)^{m_1}\left(1-\frac{\chi_{\mathcal{L}(G_2)}(x,\frac{r_2}{r_2+m_1})}{(n_2+2)(r_2+m_1)}.\frac{m_1}{x-1}\right)det\bigg[
(x-1)I_{n_1}+\frac{ J_{n_1}-I_{n_1}-A(G_1)}{n_1-1}\\&-\frac{I(G_1)}{\sqrt{(n_1-1)(n_2+2)}}\left(\frac{1}{(x-1)}{I_m}_1+\frac{\beta}{(x-1)(x-1-m_1\beta)}J_{m_1}\right)\frac{I(G_1)^T}{\sqrt{(n_1-1)(n_2+2)}}\bigg],
\end{align*} where $\beta= \frac{\chi_{\mathcal{L}(G_2)}(x,\frac{r_2}{r_2+m_1})}{(n_2+2)(r_2+m_1)}$
\begin{align*}
det S=&(x-1)^{m_1}\left(1-\beta\frac{m_1}{x-1}\right)det\bigg[
(x-1)I_{n_1}+\frac{ J_{n_1}-I_{n_1}-A(G_1)}{n_1-1}-\frac{I(G_1)I(G_1)^T}{(x-1)(n_1-1)(n_2+2)}\\&-\frac{I(G_1)}{\sqrt{(n_1-1)(n_2+2)}}\left(\frac{\beta}{(x-1)(x-1-m_1\beta)}J_{m_1}\right)\frac{I(G_1)^T}{\sqrt{(n_1-1)(n_2+2)}}\bigg]\\
=&(x-1)^{m_1}\left(1-\beta\frac{m_1}{x-1}\right)det\bigg[
(x-1)I_{n_1}+\frac{ J_{n_1}-I_{n_1}-A(G_1)}{n_1-1}\\&-\frac{I(G_1)I(G_1)^T}{(x-1)(n_1-1)(n_2+2)}-\frac{\beta}{(x-1)(x-1-m_1\beta)}\frac{I(G_1)J_{m_1}I(G_1)^T}{(n_1-1)(n_2+2)}\bigg]\\
=&(x-1)^{m_1}\left(1-\beta\frac{m_1}{x-1}\right)det\bigg[
(x-1)I_{n_1}+\frac{ J_{n_1}-I_{n_1}-A(G_1)}{n_1-1}\\&-\frac{I(G_1)I(G_1)^T}{(x-1)(n_1-1)(n_2+2)}-\frac{\beta r_1^2 J_{n_1}}{(x-1)(x-1-m_1\beta)(n_1-1)(n_2+2)}\bigg]\\
=&(x-1)^{m_1}\left(1-\beta\frac{m_1}{x-1}\right)det\bigg[
(x-1)I_{n_1}+\frac{ J_{n_1}-I_{n_1}-A(G_1)}{n_1-1}-\frac{I(G_1)I(G_1)^T}{(x-1)(n_1-1)(n_2+2)}\bigg]\\&\bigg[1-\frac{\beta r_1^2\chi_{-\frac{ J_{n_1}-I_{n_1}-A(G_1)}{n_1-1}+\frac{I(G_1)I(G_1)^T}{(x-1)(n_1-1)(n_2+2)}(x-1)} }{(x-1)(x-1-m_1\beta)(n_1-1)(n_2+2)}\bigg].\\
Thus,
\\det(S) &=(x-1)^{m_1}\left(1-\frac{n_2m_1}{(x-1)(x-\frac{m_1}{r_2+m_1})(n_2+2)(r_2+m_1)}\right)\\&\times det\bigg[
(x-1)I_{n_1}+\frac{ J_{n_1}-I_{n_1}-r_1I_{n_1}+r_1\mathcal{L}(G_1)}{n_1-1}-\frac{(2r_1I_{n_1}-r_1\mathcal{L}(G_1))}{(x-1)(n_1-1)(n_2+2)}\bigg]\\&\times\bigg[1-\frac{\frac{n_2r_1^2n_1}{(x-\frac{m_1}{r_2+m_1})(n_2+2)(r_2+m_1)(x-1+\frac{n_1-1-r_1}{n_1-1}-\frac{2r_1}{(x-1)(n_1-1)(n_2+2)}}  }{(x-1)(x-1-m_1\beta)(n_1-1)(n_2+2)}\bigg].
\end{align*}
\begin{align*}
&f(\mathcal{L}(G_1\underline{\vee} G_2),x)
=\det\bigg(xI_{n_2}-(\mathcal{L}(G_2)\circ C)\bigg)\det S\\
=&det\left((x-\frac{m_1}{r_2+m_1})I_{n_2} -\frac{r_2}{r_2+m_1}\mathcal{L}(G_2)\right)\\
&\times(x-1)^{m_1}\left(1-\frac{n_2m_1}{(x-1)(x-\frac{m_1}{r_2+m_1})(n_2+2)(r_2+m_1)}\right)\\&\times det\bigg[
(x-1)I_{n_1}+\frac{ J_{n_1}-I_{n_1}-r_1I_{n_1}+r_1\mathcal{L}(G_1)}{n_1-1}-\frac{(2r_1I_{n_1}-r_1\mathcal{L}(G_1))}{(x-1)(n_1-1)(n_2+2)}\bigg]\\&\times\bigg[1-\frac{\frac{n_2r_1^2n_1}{(x-\frac{m_1}{r_2+m_1})(n_2+2)(r_2+m_1)\left(x-1+\frac{n_1-1-r_1}{n_1-1}-\frac{2r_1}{(x-1)(n_1-1)(n_2+2)}\right)} }{(x-1)\left(x-1-\frac{m_1n_2}{(x-\frac{m_1}{r_2+m_1})(n_2+2)(r_2+m_1)}\right)(n_1-1)(n_2+2)}\bigg]\\
=&\prod_{i=2}^{n_2}\left((x-\frac{m_1}{r_2+m_1})I_{n_2} -\frac{r_2}{r_2+m_1}\tilde{\mu}_i(G_2)\right)\\
&\times(x-1)^{m_1-n_1}\bigg[(x-\frac{m_1}{r_2+m_1})-\frac{n_2m_1}{(x-1)(n_2+2)(r_2+m_1)}\bigg]\\&\prod_{i=1}^{n_1}\bigg[(x-1)^2+\frac{ \left(P(\lambda_i)-1-r_1+r_1\tilde{\mu}_i(G_1)\right)(x-1)}{n_1-1}-\frac{(2r_1-r_1\tilde{\mu}_i(G_1))}{(n_1-1)(n_2+2)}
\bigg]\\&\times\bigg[1-\frac{\frac{n_2r_1^2n_1}{(x-\frac{m_1}{r_2+m_1})(n_2+2)(r_2+m_1)(x-1+\frac{n_1-1-r_1}{n_1-1}-\frac{2r_1}{(x-1)(n_1-1)(n_2+2)}}  }{(x-1)\left(x-1-\frac{m_1n_2}{(x-\frac{m_1}{r_2+m_1})(n_2+2)(r_2+m_1)}\right)(n_1-1)(n_2+2)}\bigg]\\
=&\prod_{i=2}^{n_2}\left((x-\frac{m_1}{r_2+m_1})I_{n_2} -\frac{r_2}{r_2+m_1}\tilde{\mu}_i(G_2)\right)\\
&\times(x-1)^{m_1-n_1}\bigg[(x-\frac{m_1}{r_2+m_1})-\frac{n_2m_1}{(x-1)(n_2+2)(r_2+m_1)}\bigg]\\&\prod_{i=2}^{n_1}\bigg[(x-1)^2+\frac{ \left(-1-r_1+r_1\tilde{\mu}_i(G_1)\right)(x-1)}{n_1-1}-\frac{(2r_1-r_1\tilde{\mu}_i(G_1))}{(n_1-1)(n_2+2)}
\bigg]\\&\times\bigg[(x-1)^2+\frac{(n_1-1-r_1)(x-1)}{n_1-1}-\frac{2r_1}{(n_1-1)(n_2+2)}\\&-\frac{n_2r_1^2n_1}{(x-\frac{m_1}{r_2+m_1})(n_2+2)(r_2+m_1)\left(x-1-\frac{m_1n_2}{(x-\frac{m_1}{r_2+m_1})(n_2+2)(r_2+m_1)}\right)(n_1-1)(n_2+2)}\bigg]\\
\end{align*}
\begin{align*}
=&(x-1)^{m_1-n_1-1}\prod_{i=2}^{n_2}\left((x-\frac{m_1}{r_2+m_1})I_{n_2} -\frac{r_2}{r_2+m_1}\tilde{\mu}_i(G_2)\right)\\
&\prod_{i=2}^{n_1}\bigg[(x-1)^2+\frac{ \left(-1-r_1+r_1\tilde{\mu}_i(G_1)\right)(x-1)}{n_1-1}-\frac{(2r_1-r_1\tilde{\mu}_i(G_1))}{(n_1-1)(n_2+2)}\bigg]
\\&\times\bigg[\bigg((x-1)^2+\frac{(n_1-1-r_1)(x-1)}{n_1-1}-\frac{2r_1}{(n_1-1)(n_2+2)}\bigg)\\&\bigg((x-\frac{m_1}{r_2+m_1})(x-1)-\frac{n_2m_1}{(n_2+2)(r_2+m_1)}\bigg)-\frac{n_2r_1^2n_1}{(n_1-1)(n_2+2)^2(r_2+m_1)}\bigg].\\
\end{align*}
\end{proof}
The following corollary describes the complete normalized Laplacian spectrum of $G_1\underline{\vee} G_2$ when $G_1$ and $G_2$ are regular graphs.
\begin{cor}
Let $G_i$ be an $r_i$-regular graph with $n_i$ vertices and $m_i$ edges for i=1,2. Then the normalized Laplacian spectrum of $G_1\underline{\vee} G_2$ consists of 
\begin{enumerate}
	\item $1$ repeated $m_1-n_1-1$ times.
	\item Roots of the equation,  $$(x-\frac{m_1}{r_2+m_1}) -\frac{r_2}{r_2+m_1}\tilde{\mu}_i(G_2)=0~\textnormal{for} ~i=2,...,n_2.$$
	\item Four roots of the equation,  \begin{align*}&   \bigg((x-1)^2+\frac{(n_1-1-r_1)(x-1)}{n_1-1}-\frac{2r_1}{(n_1-1)(n_2+2)}\bigg)\\&\bigg((x-\frac{m_1}{r_2+m_1})(x-1)-\frac{n_2m_1}{(n_2+2)(r_2+m_1)}\bigg)-\frac{n_2r_1^2n_1}{(n_1-1)(n_2+2)^2(r_2+m_1)}=0.\end{align*}
	\item Two roots of the equation, $$(x-1)^2+\frac{ \left(-1-r_1+r_1\tilde{\mu}_i(G_1)\right)(x-1)}{n_1-1}-\frac{(2r_1-r_1\tilde{\mu}_i(G_1))}{(n_1-1)(n_2+2)}=0~ \textnormal{for}~ i=2,...,n_1.$$
\end{enumerate}
\end{cor}
\begin{exam}
	Let $G_1=K_{p,p}$ and  $G_2=K_2$ . Then the normalized Laplacian eigenvalues of $G_1$ are 0, 1(multiplicity $2p-2$) and 2. The normalized Laplacian eigenvalues of $G_2$ are 0 and 2. 	Therefore, the normalized Laplacian eigenvalues of $G_1\underline{\vee} G_2$ are 0, $\frac{p^2+2}{p^2+1}$,1(multiplicity $p^2-2p-1$), roots of the equation $(32p^3-16p^2+32p-16)x^3+(-112p^3+48p^2-80p+32)x^2+(120p^3-40p^2+56p-16)x+8p^2-40p^3-8p=0$, roots of the equation $(2p-1)x^2-(3p-1)x+p=0$ and roots of the equation $(8p-4)x^2-(16p-4)x+7p=0$ (each root with multiplicity $2p-2$ ).\\
	In particular, $G_1=K_{3,3}$ and  $G_2=K_2$ . Then the normalized Laplacian eigenvalues of $G_1$ are 0, 1(multiplicity 4) and 2. The normalized Laplacian eigenvalues of $G_2$ are 0 and 2. Therefore, the normalized Laplacian eigenvalues of $G_1\underline{\vee} G_2$ are 0, $\frac{11}{10}$,1(multiplicity 2), roots of the equation $800x^3-2800x^2+3032x-1032=0$, roots of the equation $5x^2-8x+3=0$ and roots of the equation $20x^2-44x+21=0$ (each root with multiplicity 4). 
\end{exam}	 
\section{Applications}
In this section, we construct a family of non-regular $\mathcal{L}$-cospectral graphs using central joins.  Also, we compute the degree Kirchhoff index and the Kemeny's constant of central graph, central vertex join and central edge join of two regular graphs. \\ In the following theorem we construct non-regular $\mathcal{L}$ cospectral graphs using central vertex join and central edge join of two regular graphs.
\begin{thm}
	Let $G_1$ and $G_2$ (not necessarily distinct) be $\mathcal{L}$-cospectral regular graphs, $H_1$ and $H_2$ (not necessarily distinct) are another $\mathcal{L}$-cospectral regular graphs. Then $\mathcal{L}(G_1\dot{\vee} H_1)$ and $\mathcal{L}(G_2\dot{\vee} H_2)$ (respectively, $\mathcal{L}(G_1\veebar H_1)$ and $\mathcal{L}(G_2\veebar H_2)$) are $\mathcal{L}$-cospectral non-regular graphs.
\end{thm}
\begin{proof}
	If $G_1$ and  $G_2$ are two $\mathcal{L}$-cospectral regular graphs, then they have the same regularity with same number of vertices and same number of edges. It is also true for $H_1$ and $H_2$. By applying Theorem 4.1 on the concerned graphs and comparing their normalized Laplacian characteristic polynomial we get the required result. 
\end{proof}
\begin{exam}
	Consider two non isomorphic regular cospectral graphs $G$ and $H$ as in \cite{van2003graphs}. Graphs $C(G)$ and $C(H)$ as shown in Figure:1. Also, $G\dot{\vee} K_2$ and $H\dot{\vee} K_2$ are non isomorphic.
	By Theorem 5.1, we obtain  $G\dot{\vee} K_2$ and $H\dot{\vee} K_2$ are $\mathcal{L} $-cospectral non-regular graphs.	
\end{exam}
	\begin{figure}[H]
\begin{figure}[H]
	\begin{minipage}[b]{0.5\linewidth}
		\centering
		\includegraphics[width=6.0 cm]{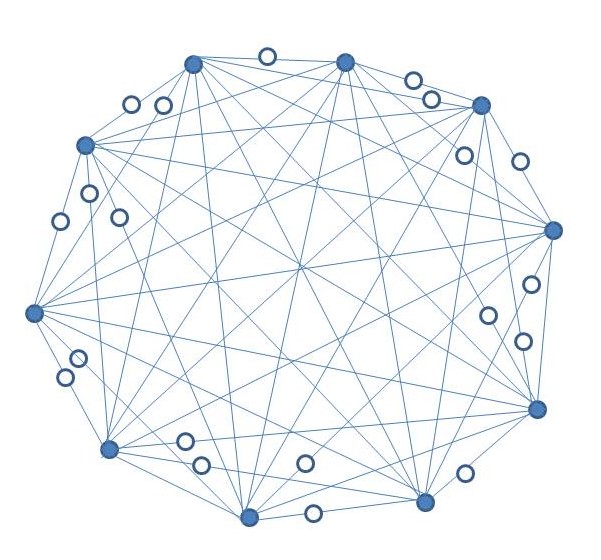}
		\caption{  $C(G)$}
		\label{pict14}
	\end{minipage}
	\hspace{0.5cm}
	\begin{minipage}[b]{0.5\linewidth}
		\centering
		\includegraphics[width=6.0 cm]{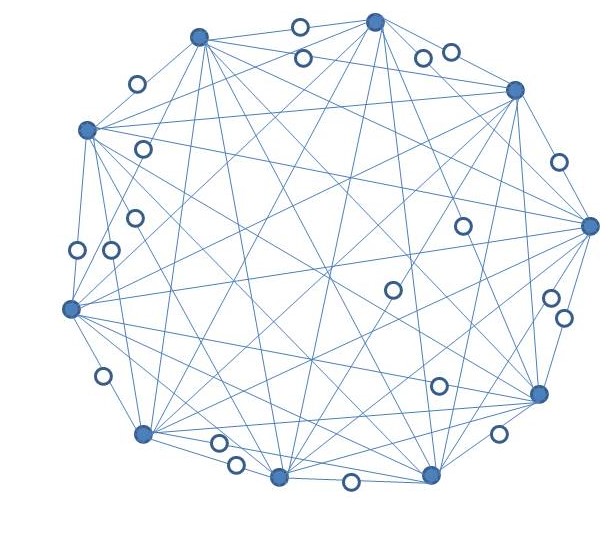}
		\caption{  $C(H)$}
		\label{pict15}
	\end{minipage}
	\caption{Figure:1 Non-regular non isomorphic $\mathcal{L} $ cospectral graphs}
\end{figure}
\end{figure}	
\begin{figure}[H]
	\begin{minipage}[b]{0.5\linewidth}
		\centering
		\includegraphics[width=8.0 cm]{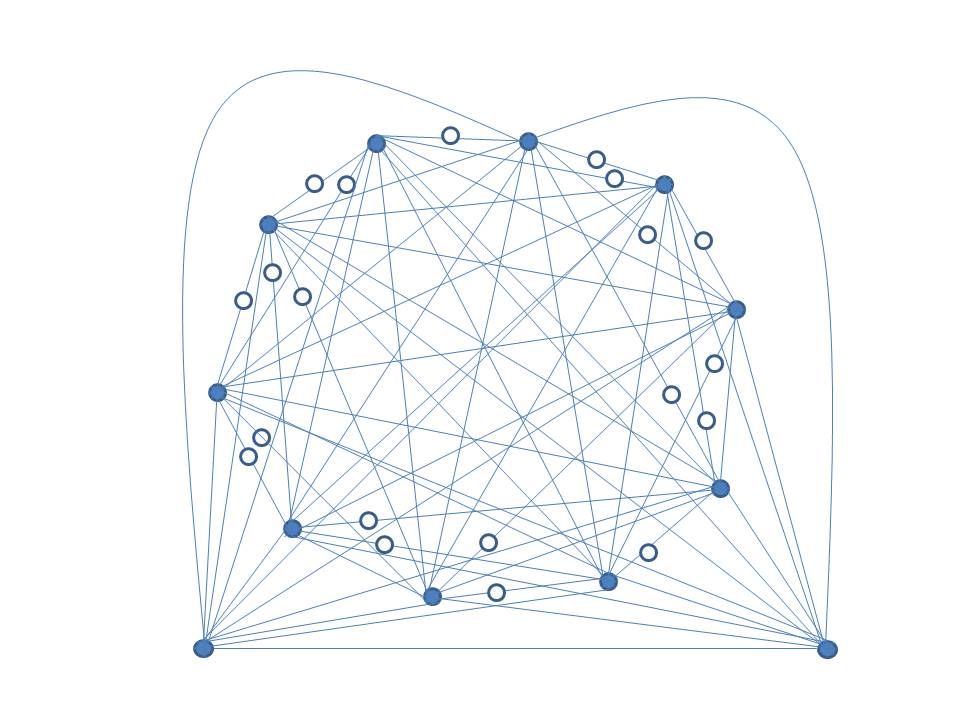}
		\caption{  $G\dot{\vee} K_2$}
		\label{pict14}
	\end{minipage}
	\hspace{0.5cm}
	\begin{minipage}[b]{0.5\linewidth}
		\centering
		\includegraphics[width=8.0 cm]{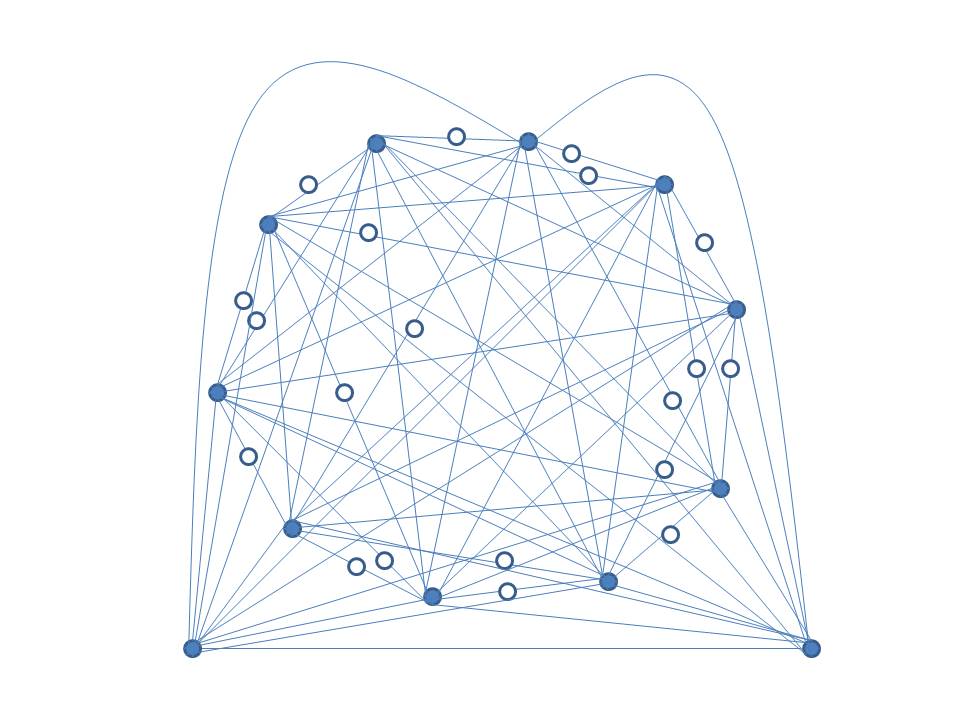}
		\caption{  $H \dot{\vee}  K_2$}
		\label{pict15}
	\end{minipage}
	\caption{Figure:2 Non-regular non isomorphic $\mathcal{L} $ cospectral graphs}
\end{figure}
By Corollary 3.2, we can readily obtain the following theorem.
\begin{thm}
Let $G$ be an r-regular graph of order  $n$ and size $m$. Then$$K(C(G))=\bigg[m-n+\frac{n-1}{n-1+r}+\sum_{i=2}^{n}\frac{2(2n+\lambda_i-1)}{2n+\lambda_i-r}\bigg].$$ 
\end{thm}
The following result is obtained from Theorem 5.2.

	Let $G$ be an r-regular graph of order  $n$ and size $m$. Then 
$$Kf^*(C(G))=(2m+n(n-1))\bigg[m-n+\frac{n-1}{n-1+r}+\sum_{i=2}^{n}\frac{2(2n+\lambda_i-1)}{2n+\lambda_i-r}\bigg].$$
By Corollary 4.2, we can readily obtain the following theorem.
\begin{thm}
 Let $G_i$ be an $r_i$-regular graph with $n_i$ vertices and $m_i$ edges  for i=1,2. Then \begin{align*}
 K( G_1\dot{\vee} G_2)=\bigg[m_1-n_1+\frac{n_1(2n_1+3n_2+r_1+r_2)+2n_2r_2-2n_1-r_2-r_1r_2}{n_1(n_1+2n_2-2)+n_2r_2}+\\\sum_{i=2}^{n_2}\frac{r_2+n_1}{n_1+r_2\tilde{\mu}_i(G_2)}+\sum_{i=2}^{n_1}\frac{2(-r_1\tilde{\mu}_i(G_1)+2n_1+2n_2+r_1-1)}{2n_1+2n_2-r_1\tilde{\mu}_i(G_1)}\bigg].
 \end{align*} 
\end{thm}
The following result is obtained from Theorem 5.3.\\
	Let $G_i$ be an $r_i$-regular graph with $n_i$ vertices and $m_i$ edges for $i=1,2.$\\ Then
\begin{align*}
Kf^*( G_1\dot{\vee} G_2)=&2m_1+2m_2+2n_1n_2+n_1(n_1-1)\\&\bigg[m_1-n_1+\frac{n_1(2n_1+3n_2+r_1+r_2)+2n_2r_2-2n_1-r_2-r_1r_2}{n_1(n_1+2n_2-2)+n_2r_2}+\\&\sum_{i=2}^{n_2}\frac{r_2+n_1}{n_1+r_2\tilde{\mu}_i(G_2)}+\sum_{i=2}^{n_1}\frac{2(-r_1\tilde{\mu}_i(G_1)+2n_1+2n_2+r_1-1)}{2n_1+2n_2-r_1\tilde{\mu}_i(G_1)}\bigg].
\end{align*}
By Corollary 4.4, we get the following theorem.
\begin{thm}
	Let $G_i$ be an $r_i$-regular graph with $n_i$ vertices and $m_i$ edges  for i=1,2. Then \begin{align*}
	&K( G_1{\veebar} G_2)=\bigg[m_1-n_1-1+\frac{r_2(4n_1n_2-n_2^2+6n_2r_1+4n_1-4n_2+4r_1-4)}{m_1(2n_2^2r_1+2n_1n_2+6n_2r_1+4n_1+2n_2)+2n_2r_1+n_2^2r_1r_2}\\&\frac{n_2^2(n_1r_2+2n_1m_1+2r_1r_2+3m_1r_1)+m_1(10n_1n_2-2n_2^2+10n_2r_1+12n_1-10n_2+8r_1-12}{m_1(2n_2^2r_1+2n_1n_2+6n_2r_1+4n_1+2n_2)+2n_2r_1+n_2^2r_1r_2}+\\&\sum_{i=2}^{n_2}\frac{r_2+m_1}{m_1+r_2\tilde{\mu}_i(G_2)}+\sum_{i=2}^{n_1}\frac{-n_2r_1\tilde{\mu}_i(G_1)-2r_1\tilde{\mu}_i(G_1)+2n_1n_2+n_2r_1+4n_1-n_2+2r_1-2}{n_1n_2+n_2r_1+2n_1-n_2r_1\tilde{\mu}_i(G_1)-3r_1\tilde{\mu}_i(G_1)}\bigg].
	\end{align*} 
\end{thm}
The following result is obtained  from Theorem 5.4.

	Let $G_i$ be an $r_i$-regular graph with $n_i$ vertices and $m_i$ edges for $i=1,2.$ Then\begin{align*}
	&Kf^*( G_1{\veebar} G_2)=2m_1+2m_2+2m_1m_2+n_1(n_1-1)\\&\bigg[m_1-n_1-1+\frac{r_2(4n_1n_2-n_2^2+6n_2r_1+4n_1-4n_2+4r_1-4)}{m_1(2n_2^2r_1+2n_1n_2+6n_2r_1+4n_1+2n_2)+2n_2r_1+n_2^2r_1r_2}\\&\frac{n_2^2(n_1r_2+2n_1m_1+2r_1r_2+3m_1r_1)+m_1(10n_1n_2-2n_2^2+10n_2r_1+12n_1-10n_2+8r_1-12)}{m_1(2n_2^2r_1+2n_1n_2+6n_2r_1+4n_1+2n_2)+2n_2r_1+n_2^2r_1r_2}+\\&\sum_{i=2}^{n_2}\frac{r_2+m_1}{m_1+r_2\tilde{\mu}_i(G_2)}+\sum_{i=2}^{n_1}\frac{-n_2r_1\tilde{\mu}_i(G_1)-2r_1\tilde{\mu}_i(G_1)+2n_1n_2+n_2r_1+4n_1-n_2+2r_1-2}{n_1n_2+n_2r_1+2n_1-n_2r_1\tilde{\mu}_i(G_1)-3r_1\tilde{\mu}_i(G_1)}\bigg].
	\end{align*}

\section{Conclusion}
 In this paper, the normalized Laplacian spectrum of central graph,  central vertex join and central edge join of two regular graphs are calculated. Also, we construct non-regular $\mathcal{L}$-cospectral  graphs. Moreover, we investigate some graph invariants like the degree Kirchhoff index and the Kemeny's constant   of central graph,  central vertex join and central edge join of two regular graphs.
\bibliography{refecentral}
\bibliographystyle{unsrt}
\end{document}